\documentclass{amsproc}
\usepackage{amsmath}
\usepackage{enumerate}
\usepackage{amsmath,amsthm,amscd,amssymb}

\usepackage{latexsym}
\usepackage{upref}
\usepackage{verbatim}
\usepackage{pstricks}
\usepackage[mathscr]{eucal}
\usepackage{dsfont}

\usepackage{graphicx}
\usepackage[colorlinks,hyperindex,pagebackref,hypertex]{hyperref}

\usepackage{hhline}

\usepackage[OT2,OT1]{fontenc}
\newcommand\cyr
{
\renewcommand\rmdefault{wncyr}
\renewcommand\sfdefault{wncyss}
\renewcommand\encodingdefault{OT2}
\normalfont
\selectfont
}
\DeclareTextFontCommand{\textcyr}{\cyr}

\newtheorem{theorem}{Theorem}

\newtheorem{corollary}[theorem]{Corollary}

\newtheorem{definition}[theorem]{Definition}
\newtheorem{remark}[theorem]{Remark}

\chardef\bslash=`\\ 

\newcommand{\bbR}{{\mathbb{R}}}

\newcommand{\bbC}{{\mathbb{C}}}

\newcommand{\ran}{\text{\rm{Ran}}}

\newcommand{\calD}{{\mathcal D}}

\newcommand{\calH}{{\mathcal H}}

\newcommand{\calN}{{\mathcal N}}

\newcommand{\calS}{{\mathcal S}}

\renewcommand{\Im}{\text{\rm Im}}

\def\sM{{\mathfrak M}}

      \def\dC{{\mathbb C}}

   \def\cN{{\mathcal N}}   
      
\def\cS{{\mathcal S}}

\def\RE{{\rm Re\,}}

\DeclareMathOperator{\IM}{Im}

\newcommand{\eval}[2][\right]{\relax
  \ifx#1\right\relax \left.\fi#2#1\rvert}

\begin{document}

\title{L-systems with Multiplication Operator and c-Entropy}

\author{S. Belyi}
\address{Department of Mathematics\\ Troy University\\
Troy, AL 36082, USA\\
}
\curraddr{}
\email{sbelyi@troy.edu}


\author[Makarov]{K. A. Makarov}
\address{Department of Mathematics\\
 University of Missouri\\
  Columbia, MO 63211, USA}
\email{makarovk@missouri.edu}


\author{E. Tsekanovskii}
\address{Department of Mathematics, Niagara University,  Lewiston,
NY  14109, USA} \email{\tt tsekanov@niagara.edu}


\date{DD/MM/2004}



\begin{abstract}

In this note, we utilize the concepts of c-entropy and the dissipation coefficient in connection with canonical L-systems based on the multiplication (by a scalar) operator. Additionally, we examine the coupling of such L-systems and derive explicit formulas for the associated  c-entropy and dissipation coefficient.  In this context,  we also  introduce the concept of a skew-adjoint L-system  and analyze its coupling with the original L-system.

 \end{abstract}

\maketitle

\tableofcontents


\section{Introduction}\label{s1}

This paper is dedicated to the exploration of the relationships between various subclasses of Herglotz-Nevanlinna functions and their conservative realizations as impedance functions of L-system (see \cite{ABT,BMkT,BMkT-2,BMkT-3,BT-21,Lv2}).

Recall the concept of a canonical L-system.

Let $T$ be a bounded linear operator in a Hilbert space $\calH$. Suppose, in addition, that
 $E$ is another  finite-dimensional  Hilbert space, $\dim E<\infty$.
By  a  \textit{canonical L-system} we mean the array
\begin{equation}
\label{col0}
 \Theta =
\left(%
\begin{array}{ccc}
  T    & K & J \\
   \calH&  & E \\
\end{array}%
\right),
\end{equation}
where $K\in[E,\calH]$,  $J$ is a self-adjoint isometry in $E$ such that   $\IM T=KJK^*$.

Recall that the operator-valued function
\begin{equation*}\label{W1}
 W_\Theta(z)=I-2iK^*(T-zI)^{-1}KJ,\quad z\in \rho(T),
\end{equation*}
 is called the \textit{transfer function}  of an L-system $\Theta$, while
\begin{equation*}\label{real2}
 V_\Theta(z)=i[W_\Theta(z)+I]^{-1}[W_\Theta(z)-I]J =K^*(\RE T-zI)^{-1}K,\quad z\in\rho(T)\cap\dC_{\pm},
\end{equation*}
is {}{known  as}  the \textit{impedance function } of $\Theta$. The formal definition of  L-systems and their elements are presented in Section \ref{s2}.

The main goal of this note is to  apply the concepts of  c-entropy and dissipation coefficient
{}{ introduced in \cite{BT-16} and \cite{BT-21} }
to canonical  L-systems based on the multiplication operator.

The paper is organized as follows.

Section \ref{s2} contains necessary information on the L-systems theory.

In Section \ref{s3}, we review the formal definition and discuss the fundamental properties of regular and generalized Donoghue classes with {a} bounded,  compactly supported  representing measure. In particular, we emphasize the connection between the analytical theory of electrical circuits and the subclasses of rational functions belonging to the Donoghue classes (see Remark \ref{zepi}).

Sections \ref{s4}--\ref{s7} of the paper contain the  main results.

Section \ref{s4} {}{describes} a detailed construction of a canonical L-system associated with the multiplication by a scalar operator.  In this section, we also establish criteria for the impedance function of such an L-system to belong to the Donoghue classes with { a bounded representing measure.}

In Section \ref{s5}, we revisit the concept of L-system coupling  (see \cite{Bro}, \cite{ABT}), and apply it to the canonical L-systems with  the multiplication operators introduced in Section \ref{s4}. We derive the explicit form of the coupling between two canonical L-systems (with multiplication operators) and {}{give an independent proof of  the Multiplication Theorem  (see, e.g., \cite{Bro,MT10}) stating that the transfer function of this coupling is equal to the product of the transfer functions associated with the factor L-systems.
As far as  the corresponding impedance functions  are concerned, we  demonstrate that when the coupling of two L-systems, whose impedance functions belong to  Donoghue classes, is formed, then  the impedance of the resulting L-system also belongs to the Donoghue class. Thus, the invariance {(persistence)} of the Donoghue classes with respect to the coupling operation is established.

In Section \ref{s6}, we revisit the definition of c-entropy and the dissipation coefficient, focusing specifically on their application to canonical L-systems with multiplication operators.
We also obtain explicit representations for {both  c-entropy} and the dissipation coefficient, and show that c-entropy is additive with respect to the coupling of two canonical L-systems (with multiplication operators).

In Section \ref{s7}, we introduce the skew-adjoint L-systems associated with the canonical L-systems defined by multiplication operators and examine the properties of c-entropy and dissipation coefficient in this setting (also see Remark \ref{zepi2} for the  ``inductor-capacitor network'' interpretation for  the coupling between  two elementary L-systems).

All the new results {presented  are} summarized in Tables \ref{Table-1} and \ref{Table-2}. The paper concludes with illustrative examples that demonstrate the constructions discussed.

\section{Preliminaries}\label{s2}

For a pair of Hilbert spaces $\calH_1$, $\calH_2$ denote by $[\calH_1,\calH_2]$ the set of all bounded linear operators from $\calH_1$ to $\calH_2$.

Let $T$ be a bounded linear operator in a Hilbert space $\calH$, $K\in[E,\calH]$, and $J$ be a bounded, self-adjoint, and unitary operator in $E$, where $E$ is another Hilbert space with $\dim E<\infty$. Let also
\begin{equation}\label{e5-3}
\IM T=KJK^*.
\end{equation}

By definition, the  \textbf{Liv\v{s}ic canonical
system } or simply \textbf{canonical L-system} is just the  array
\begin{equation}
\label{BL}
 \Theta =
\left(%
\begin{array}{ccc}
  T & K & J \\
  \calH &  & E \\
\end{array}%
\right).
 \end{equation}
 The spaces $\calH$ and $E$ here  are
called \textit{state} and
\textit{input-output} spaces, and the operators
$T$, $K$, $J$ will be refered to as \textit{main},
\textit{channel}, and \textit{directing} operators, respectively.

Notice that
relation
\eqref{e5-3}
implies
\begin{equation}\label{e5-4}
    \ran (\IM T)\subseteq\ran(K).
\end{equation}

We  associate with an L-system $\Theta$ two  analytic functions,  the \textbf{transfer  function} of the L-system $\Theta$
\begin{equation}\label{e6-3-3}
W_\Theta (z)=I-2iK^\ast (T-zI)^{-1}KJ,\quad z\in \rho (T),\quad {\text{the resolvent set, }}
\end{equation}
and also the \textbf{impedance function}  given by the formula
\begin{equation}\label{e6-3-5}
V_\Theta (z) = K^\ast (\RE T - zI)^{-1} K, \quad z\in  \rho (\RE T).
\end{equation}

 The transfer function $W_\Theta (z)$ of the L-system $\Theta $ and function $V_\Theta (z)$ of the form (\ref{e6-3-5}) are connected by the following relations valid for $\IM z\ne0$, $z\in\rho(T)$,
\begin{equation}\label{e6-3-6}
\begin{aligned}
V_\Theta (z) &= i [W_\Theta (z) + I]^{-1} [W_\Theta (z) - I]J,\\
W_\Theta(z)&=(I+iV_\Theta(z)J)^{-1}(I-iV_\Theta(z)J).
\end{aligned}
\end{equation}

Recall that  the  impedance function $V_\Theta(z)$ of a canonical L-system admits the  integral representation (see, e.g.,  \cite[Section 5.5]{ABT}, \cite{Bro})
\begin{equation}\label{hernev-real}
V_\Theta(z)=\int_\bbR \frac{d\sigma(t)}{{t}-z},
\end{equation}
where  $\sigma$ is an   operator-valued bounded Borel measure in $E$ with the compact support on $\bbR$.

As far as the inverse problem is concerned,  we refer to  \cite{ABT,BMkT,GT} and references therein for the description
of the class of all Herglotz-Nevanlinna functions that admit  realizations as impedance  functions of an L-system.

\section{Bounded Donoghue classes}\label{s3}

Denote by  $\widehat{\calN}$  the  class of all scalar Herglotz-Nevanlinna functions $M(z)$ that admit the representation \eqref{hernev-real}
where $\sigma$ is a compactly supported  bounded Borel measure such that
\begin{equation}\label{e-9-cond}
\int_\bbR\frac{t}{1+{t}^2}\,d\sigma({t})=0.
\end{equation}
Observe that for a function $M(z)\in\widehat{\calN}$ condition \eqref{e-9-cond} implies
$$
M(i)=\int_\bbR \frac{d\sigma(t)}{t-i}={\int_\bbR \frac{t}{t^2+1}\,d\sigma(t)+i\int_\bbR \frac{d\sigma(t)}{t^2+1}}=i\int_\bbR \frac{d\sigma(t)}{t^2+1}.
$$

 Following the approach of  \cite{BMkT,BMkT-3,MT10,MT2021}, given   $0\le\kappa<1$,  denote by $\widehat{\sM}$,  $\widehat{\sM_\kappa}$ and  $\widehat{\sM_\kappa^{-1}}$  the subclass of $\widehat{\calN}$     satisfying the additional requirement  that
\begin{equation}\label{e-42-int-don}
\int_\bbR\frac{d\sigma({t})}{1+{t}^2}=1\,,\quad\text{equivalently,}\quad M(i)=i,
\end{equation}
\begin{equation}\label{e-38-kap}
\int_\bbR\frac{d\sigma({t})}{1+{t}^2}=\frac{1-\kappa}{1+\kappa}\,,\quad\text{equivalently,}\quad M(i)=i\,\frac{1-\kappa}{1+\kappa},
\end{equation}
and
\begin{equation}\label{e-39-kap}
\int_\bbR\frac{d\sigma({t})}{1+{t}^2}=\frac{1+\kappa}{1-\kappa}\,,\quad\text{equivalently,}\quad M(i)=i\,\frac{1+\kappa}{1-\kappa},
\end{equation}
respectively. Clearly,  $$\widehat{\sM}=\widehat{\sM_0}=\widehat{\sM_0^{-1}}.$$

It is easy to see that  if $M(z)$ is an arbitrary function from the class $\widehat{\calN}$ such that
\begin{equation}\label{e-66-L}
\int_\bbR\frac{d\sigma({t})}{1+{t}^2}=a \quad\text{for some}\quad  a>0,
\end{equation}
 then  $M\in\widehat{\sM}$ if and only if $a=1$.

The membership of $M\in \widehat{\cN}$ in the other generalized Donoghue classes   $ \widehat{\sM_\kappa} $  and   $\widehat{\sM_\kappa^{-1}}$ can
also be described straightforwardly as follows:
 \begin{enumerate}
\item[] if $a<1$, then $M\in \widehat{\sM_\kappa}$ with
\begin{equation}\label{e-45-kappa-1}
\kappa=\frac{1-a}{1+a},
\end{equation}
\item[]and
\item[]if $a>1$, then $M\in \widehat{\sM_\kappa^{-1}}$ with
\begin{equation}\label{e-45-kappa-2}
\kappa=\frac{a-1}{1+a}.
 \end{equation}
 \end{enumerate}

\begin{remark}\label{zepi} We note that rational functions from Donoghue classes
{}{taking purely imaginary values on the imaginary axis} can be characterized as follows.
Suppose  that a rational function  is given by
$$M(z)=-\frac{a_0}z+\sum_{k=1}^n a_k\frac{z}{b_k^2-z^2},\quad z\in \bbC_+,
$$
where $a_0\ge 0$ and  $a_k>0$, $b_k>0$, $k=1,\dots, n$.
Clearly, the function $M(z)$ admits the representation \eqref{hernev-real}
$$
M(z)=\int_{\bbR}{\frac{d\sigma(t)}{t-z}},
$$
where $\sigma(dt)$ is a singular (atomic) measure supported at the points $-b_n, \dots,-b_1$, $0$, $b_1,\dots, b_n$ with the weights  $\frac 12 a_n, \dots,\frac 12 a_1,a_0, \frac 12 a_1,\dots, \frac 12 a_n$, respectively.
In particular,   condition \eqref{e-9-cond} holds, that is,
$$
\int_\bbR\frac{t}{1+{t}^2}\,d\sigma({t})=0.
$$
Moreover,
$$M(i)=i\left (a_0+\sum_{k=1}^n\frac{a_k}{b_k^2+1}\right)\in i\mathbb{R}_+.
$$
Therefore,  the rational function $M(z)$ belongs to one of  the Donoghue classes
 $\widehat{\sM}$,  $\widehat{\sM_\kappa}$ and  $\widehat{\sM_\kappa^{-1}}$ depending on whether
 $$a=a_0+\sum_{k=1}^n\frac{a_k}{b_k^2+1}$$
 equals to, less than, or greater than one, respectively.
 Notice that the rational function
$$
Z(p)=\frac{1}{i}M(ip)=\frac{a_0}{p}+\sum_{k=1}^n  a_k\frac{p}{b_k^2+p^2}$$
is a positive analytic function in the right half-plane,
that is, $$Z(p)+\overline{Z(p)}>0, \quad \RE(p)>0,
$$
takes real values on the positive semi-axis and has  purely imaginary values on the imaginary axis
(for  the definition of positive functions and their role in designing
 (synthesis of) electrical circuits, we refer to \cite{EfimPotap73}).

Since
$$
Z(p)=
\frac{1}{\frac1{a_0} p}+\sum_{k=1}^n
\frac{1}
{\left (L_kp\right )^{-1}
+\left (\frac{1}{C_kp}\right )^{-1}},
$$
where
$$L_k=\frac{a_k}{b_k^2} \quad \text{and}\quad C_k=\frac{1}{a_k},
$$
the  function $Z(p)$ determines the impedance of a series connection of a capacitor $C_0=\frac1{a_0}$ and $n$ parallel $LC$-circuits
with inductance $L_k$ and capacitance $C_k$, $k=1,\dots, n$. We remark that if $a_0=0$,
  the purely capacitive element of the circuit is  absent in this case.   Figure \ref{fig-2} shows an electric circuit consisting of a single capacitor and two $LC$-circuits connected in series.
\end{remark}

\begin{figure}
  \begin{center}
 \includegraphics[width=90mm]{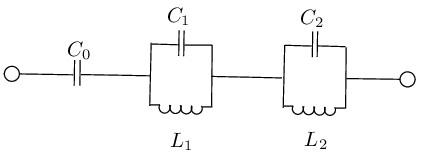}
  \caption{Three-stage Series Electrical Circuit }\label{fig-2}
  \end{center}
\end{figure}

\section{L-systems with multiplication operator}\label{s4}

Let $\calH$ be a {\it one} dimensional Hilbert space with an inner product $(\cdot,\cdot)$ and a normalized basis vector $h_0\in \calH$, ($\|h_0\|=1$). For a fixed number $\lambda_0\in\dC$ with $\IM\lambda_0>0$ we introduce (see also \cite{Bro}) a linear operator
\begin{equation}\label{e-d-4-30}
    T h=\lambda_0 h,\quad h\in\calH.
\end{equation}
Clearly,
$$
T^* h=\bar\lambda_0 h,\quad\IM T h=\IM\lambda_0 h,\quad\textrm{ and }\quad \RE T h=\RE \lambda_0 h,\quad \forall h\in\calH.
$$
We are going to include $T$ into a simple L-system $\Theta$, i.e., construct an L-system where $\calH$ is the state space and $T$ is the main operator (see \cite{Bro}, \cite{BT-21}). Let $K:\dC\rightarrow\calH$ be such that
\begin{equation}\label{e-4-31}
   K c=(\sqrt{\IM\lambda_0}\, c) h_0,\quad \forall c\in\dC.
\end{equation}
Then, the adjoint operator $K^*:\calH\rightarrow\dC$ is
\begin{equation}\label{e-4-32}
   K^* h=(h, \sqrt{\IM\lambda_0}\,h_0),\quad \forall h\in\calH.
\end{equation}
Note that for an arbitrary element $h=C_h h_0\in\calH$, ($C_h\in\dC$) formulas \eqref{e-4-31} and \eqref{e-4-32} imply
$$
 K^* h= K^* (C_h h_0)=C_h\sqrt{\IM\lambda_0}\quad \mathrm{and} \quad K\, C_h=(\sqrt{\IM\lambda_0})h.
$$
Therefore,
$$
K\,K^* h=K(C_h\sqrt{\IM\lambda_0})=(\IM\lambda_0)h.
$$
  Furthermore,
\begin{equation}\label{e-4-33}
    \IM T h=\IM\lambda_0 h=(\sqrt{\IM\lambda_0}) (h, \sqrt{\IM\lambda_0}\,h_0)h_0=K\,K^* h,\; \forall h\in\calH.
\end{equation}
Thus, we can construct an L-system of the form
\begin{equation}\label{e-4-34}
    \Theta= \begin{pmatrix} T&K&\ 1\cr  \calH & &\dC\cr \end{pmatrix},
\end{equation}
where operators $T$ and $K$ are defined by \eqref{e-d-4-30} and \eqref{e-4-31}, respectively. Taking into account that
$$
(T-zI)^{-1}=\frac{1}{\lambda_0-zI},
$$
for  the transfer function $W_\Theta(z)$ we have
\begin{equation}\label{e-4-35}
    W_\Theta (z)=I-2iK^\ast (T-zI)^{-1}K=1-2i\frac{\IM \lambda_0}{\lambda_0-z}=\frac{\bar\lambda_0-z}{\lambda_0-z}.
\end{equation}
The corresponding impedance function  \eqref{e6-3-5} is easily found to be   given by
\begin{equation}\label{e-4-36}
    V_\Theta (z)=K^\ast (\RE T - zI)^{-1} K=\frac{\IM\lambda_0}{\RE\lambda_0-z}.
\end{equation}
By inspection, one confirms that $V_\Theta (z)$ is a Herglotz-Nevanlinna function.

The following result provides conditions on when the impedance function of the form \eqref{e-4-36} belongs to the Donoghue classes described in Section \ref{s3}.
\begin{theorem}\label{t-6}
Let $\Theta$ be an L-system \eqref{e-4-34} with the main operator $T$ of the form \eqref{e-d-4-30}. Then  the impedance function $V_{\Theta}(z)$ belongs to the class:
\begin{enumerate}
  \item  $\widehat{\sM}$ if and only if $\lambda_0=i$;
  \item $\widehat{\sM}_\kappa$ if and only if $\lambda_0=ai$, $0<a<1$;
  \item $\widehat{\sM}_\kappa^{-1}$ if and only if $\lambda_0=ai$, $a>1$.
\end{enumerate}
\end{theorem}
\begin{proof}
Let us prove (1) first. It is well known (see \cite{ABT}) that the impedance function $V_\Theta(z)$ of any L-system is a Herglotz-Nevanlinna function and hence has integral representation \eqref{hernev-real}. Consequently,
$V_\Theta(z)$ of the form \eqref{e-4-36}  belongs to the class $\widehat{\sM}$ if and only if $V_\Theta(i)=i$ (condition \eqref{e-9-cond} is satisfied automatically). Using \eqref{e-4-36} yields
\begin{equation}\label{e-4-37}
 V_\Theta(i)=\frac{\IM\lambda_0}{\RE\lambda_0-i}=\frac{\IM\lambda_0\cdot\RE\lambda_0+i\IM\lambda_0}{(\RE\lambda_0)^2+1}
 =\frac{\IM\lambda_0\cdot\RE\lambda_0}{(\RE\lambda_0)^2+1}+i\frac{\IM\lambda_0}{(\RE\lambda_0)^2+1}.
\end{equation}
This expression on the right equals $i$ if and only if $\RE\lambda_0=0$ and $\IM\lambda_0=1$, i.e., $\lambda_0=i$. This proves (1).

In order to prove (2), we compare \eqref{e-4-37} to formulas \eqref{e-38-kap}, \eqref{e-39-kap}, and \eqref{e-66-L}. This comparison yields that $V_\Theta(z)$ of the form \eqref{e-4-36} belongs to the class $\widehat{\sM}_\kappa$ if
$$
V_\Theta(i)=ai=(\IM\lambda_0)i, \quad 0<a=\IM\lambda_0<1.
$$
Therefore, $V_\Theta(z)\in\widehat{\sM}_\kappa$ if and only if $\RE\lambda_0=0$ and $\IM\lambda_0<1$. Thus, $\lambda_0=ai$, $0<a<1$. Also, in this case
\begin{equation}\label{e-4-38}
    \kappa=\frac{1-\IM\lambda_0}{1+\IM\lambda_0}.
\end{equation}

Assertion (3) is proved similarly to (2) with an assumption that $a=\IM\lambda_0$ is greater than 1. Moreover, here we have $V_\Theta(z)\in\widehat{\sM}_\kappa^{-1}$ with
\begin{equation}\label{e-4-39}
    \kappa=\frac{\IM\lambda_0-1}{1+\IM\lambda_0}.
\end{equation}
\end{proof}

\begin{table}[ht]
\centering
\begin{tabular}{|c|c|c|c|}
\hline
 &  &  &\\
 \textbf{$\lambda_0$}& $a$  & $\kappa$  & \textbf{Class}  \\
  &  & &\\ \hline
&  &  &\\
  $i$ & $a=1$ & $\kappa=0$ &$V_{\Theta}(z)\in\widehat{\sM}$\\
  &  &  & \\  \hline
 &  &  &\\
 $ai$&  $0<a<1$& $\kappa=\frac{1-a}{1+a}$ &$V_{\Theta}(z)\in\widehat{\sM}_\kappa$\\
 &  &  & \\
  \hline
 &  &  &\\
  $ai$ &  $a>1$& $\kappa=\frac{a+1}{1+a}$ &$V_{\Theta}(z)\in\widehat{\sM}^{-1}_\kappa$\\
  &  &  & \\
     \hline
     \multicolumn{1}{l}{} & \multicolumn{1}{l}{} & \multicolumn{1}{l}{} & \multicolumn{1}{l}{}
\end{tabular}
\caption{Impedance function $V_{\Theta}(z)$}
\label{Table-1}
\end{table}

The results of Theorem  \ref{t-6}  are summarized in  Table \ref{Table-1}.

\section{L-system coupling}\label{s5}

In this section, following \cite{Bro} (see also \cite{ABT} and \cite{BMkT-2}), we  introduce the coupling of two L-systems with multiplication operator of the form \eqref{e-4-34}.

 Let (as above) $\calH$ be  a one-dimensional Hilbert space with the inner product $(\cdot,\cdot)$ and a normalized basis vector $h_0\in \calH$, ($\|h_0\|=1$). Suppose, in addition, that    $T_1$ and $T_2$ are  multiplication operators in $\calH$ of the form \eqref{e-d-4-30} given by
 \begin{equation}\label{e-45-T12}
    T_1 h=\lambda_0 h\quad \textrm{and}\quad T_2 h=\mu_0 h,\quad h\in\calH.
\end{equation}
Consider two L-systems of the form \eqref{e-4-34} based on $T_1$ and $T_2$, respectively,
\begin{equation}\label{e-46-Theta12}
    \Theta_1= \begin{pmatrix} T_1&K_1&\ 1\cr  \calH & &\dC\cr \end{pmatrix}
\quad \textrm{and}\quad
\Theta_2= \begin{pmatrix} T_2&K_2&\ 1\cr  \calH & &\dC\cr \end{pmatrix},
\end{equation}
where $K_1:\dC\rightarrow\calH$ and $K_2:\dC\rightarrow\calH$
are such that
\begin{equation}\label{e-47-K12}
   K_1 c=(\sqrt{\IM\lambda_0}\, c) h_0,
   \quad \textrm{and}\quad
    K_2 c=(\sqrt{\IM\mu_0}\, c) h_0,
   \quad \forall c\in\dC.
\end{equation}
Recall \eqref{e-4-32} that  the adjoint operators $K_1^*:\calH\rightarrow\dC$ and $K_2^*:\calH\rightarrow\dC$ are
\begin{equation}\label{e-48-adj}
   K^*_1 h=(h, \sqrt{\IM\lambda_0}\,h_0),\quad K^*_2 h=(h, \sqrt{\IM\mu_0}\,h_0),\quad\forall h\in\calH.
\end{equation}
Define an operator $\mathbf{T}$ acting on the Hilbert space $\calH\oplus\calH$ as
\begin{equation}\label{e-48-bT}
  \begin{aligned}
   \mathbf{T}\left[
               \begin{array}{c}
                 h_1 \\
                 h_2 \\
               \end{array}
             \right]&=\left(
                       \begin{array}{cc}
                         T_1 & 2iK_1K_2^* \\
                         0 & T_2 \\
                       \end{array}
                     \right)\left[ \begin{array}{c}
                 h_1 \\
                 h_2 \\
               \end{array}
             \right]=\left[
               \begin{array}{c}
                 T_1 h_1+2i K_1K_2^*h_2 \\
                 T_2 h_2 \\
               \end{array}
             \right]\\
             &=\left[
               \begin{array}{c}
                 \lambda_0 h_1+2i  \sqrt{(\IM\lambda_0)\cdot(\IM\mu_0)}\,h_2\\
                 \mu_0 h_2 \\
               \end{array}
             \right],\quad\forall h_1,h_2\in\calH.
      \end{aligned}
   \end{equation}
Direct check reveals that
 \begin{equation}\label{e-50-bT-adj}
  \begin{aligned}
   \mathbf{T^*}\left[
               \begin{array}{c}
                 h_1 \\
                 h_2 \\
               \end{array}
             \right]&=\left(
                       \begin{array}{cc}
                         T_1^* & 0 \\
                         -2iK_2K_1^* & T_2^* \\
                       \end{array}
                     \right)\left[ \begin{array}{c}
                 h_1 \\
                 h_2 \\
               \end{array}
             \right]=\left[
               \begin{array}{c}
                 T_1^* h_1 \\
                 T_2^* h_2-2i K_2K_1^*h_1 \\
               \end{array}
             \right]\\
             &=\left[
               \begin{array}{c}
             \bar\lambda_0 h_1    \\
               \bar\mu_0 h_2-2i  \sqrt{(\IM\lambda_0)\cdot(\IM\mu_0)}\,h_1   \\
               \end{array}
             \right],\quad\forall h_1,h_2\in\calH.
      \end{aligned}
   \end{equation}
In addition to $\mathbf{T}$ we define an operator $\mathbf{K}:\dC\rightarrow\calH\oplus\calH$ as
\begin{equation}\label{e-50-bK}
    \mathbf{K}c=\left[
               \begin{array}{c}
                 K_1 c \\
                 K_2 c\\
               \end{array}
             \right]=\left[
               \begin{array}{c}
                 (\sqrt{\IM\lambda_0}\, c) h_0 \\
                (\sqrt{\IM\mu_0}\, c) h_0\\
               \end{array}
             \right].
\end{equation}
In this case the adjoint operator $\mathbf{K}^*:\calH\oplus\calH\rightarrow\dC$ is defined for all $h_1,h_2\in\calH$ as follows
\begin{equation}\label{e-51-b-adj}
    \mathbf{K}^*\left[
               \begin{array}{c}
                 h_1 \\
                 h_2 \\
               \end{array}
             \right]=K_1^*h_1+K_2^*h_2=(h_1, \sqrt{\IM\lambda_0}\,h_0)+(h_2, \sqrt{\IM\mu_0}\,h_0).
\end{equation}
One can confirm that
\begin{equation}\label{e-53-ImbT}
     \begin{aligned}
    \IM \mathbf{T}\left[
               \begin{array}{c}
              h_1     \\
               h_2     \\
               \end{array}
             \right]&=
             \frac{1}{2i}\left(
                       \begin{array}{cc}
                         2i\IM T_1 & 2iK_1K_2^* \\
                         2iK_2K_1^* & 2i\IM T_2 \\
                       \end{array}
                     \right)
             \left[
               \begin{array}{c}
              h_1     \\
               h_2     \\
               \end{array}
             \right]\\
             &=\left[
               \begin{array}{c}
                K_1K_1^* h_1+K_1K_2^* h_2 \\
                K_2K_2^* h_2+K_2K_1^* h_1\\
               \end{array}
             \right]=\mathbf{K}\mathbf{K}^*\left[
               \begin{array}{c}
                 h_1 \\
                 h_2 \\
               \end{array}
             \right],
        \end{aligned}
\end{equation}
and hence $\IM \mathbf{T}=\mathbf{K}\mathbf{K}^*$.

{Summarizing, we arrive at the following definition.}
\begin{definition}\label{d-11}
Given two systems of the form \eqref{e-46-Theta12}
$$
  \Theta_1= \begin{pmatrix} T_1&K_1&\ 1\cr  \calH & &\dC\cr \end{pmatrix}
\quad \textrm{and}\quad
\Theta_2= \begin{pmatrix} T_2&K_2&\ 1\cr  \calH & &\dC\cr \end{pmatrix},
$$
define the  \textbf{coupling of two L-systems} as
 \begin{equation*}
\Theta= \begin{pmatrix} \mathbf{T}&\mathbf{K}&\ 1\cr \calH\oplus\calH& &\dC\cr
\end{pmatrix},
\end{equation*}
where the operators $\mathbf{T}$ and $\mathbf{K}$  are presented in \eqref{e-48-bT}--\eqref{e-51-b-adj}.

In writing,
$$
\Theta=\Theta_1\cdot\Theta_2.
$$
\end{definition}

The following theorem can be derived from a more general result \cite[Theorem 3.1]{Bro},
{however, for {the} convenience of the reader, we provide a simple  proof of it.}
 \begin{theorem}[cf. \cite{Bro}]\label{unitar}
 Let an L-system $\Theta$ be the coupling of two L-systems $\Theta_1$
and $\Theta_2$ of the form \eqref{e-46-Theta12} with the main operators $T_1$ and $T_2$ given by \eqref{e-45-T12}. Then if $z\in\rho(T_1)\cap\rho(T_2)={\bbC\setminus\{\lambda_0,\mu_0\}}$, we have
\begin{equation}\label{e-91-mult}
W_\Theta(z)=W_{\Theta_1}(z)\cdot W_{\Theta_2}(z)=\frac{\bar\lambda_0-z}{\lambda_0-z}\cdot\frac{\bar\mu_0-z}{\mu_0-z}.
\end{equation}
\end{theorem}
\begin{proof}
We are going to find the transfer function of the coupling $\Theta$ based on its definition \eqref{e6-3-3} and then show that it equals the right hand side of \eqref{e-91-mult}. Using \eqref{e-48-bT} and elementary calculations one finds
\begin{equation}\label{e-47-res}
    \begin{aligned}
(\mathbf{T}-zI)^{-1}&=\left(
                       \begin{array}{cc}
                         \lambda_0-z & 2i\sqrt{\IM\lambda_0\cdot\IM\mu_0} \\
                         0 & \mu_0-z \\
                       \end{array}
                     \right)^{-1}\\
                     &=\frac{1}{(\lambda_0-z)(\mu_0-z)}\left(
                       \begin{array}{cc}
                         \mu_0-z & -2i\sqrt{\IM\lambda_0\cdot\IM\mu_0} \\
                         0 & \lambda_0-z \\
                       \end{array}
                     \right).
    \end{aligned}
\end{equation}
Taking in to account that
$$
2i\sqrt{\IM\lambda_0\cdot\IM\mu_0}=\sqrt{(\lambda_0-\bar\lambda_0)(\mu_0-\bar\mu_0)},
$$
we use \eqref{e-50-bK}, \eqref{e-51-b-adj} with \eqref{e-47-res} and  proceed to find
$$
    \begin{aligned}
2i\mathbf{K}^*(\mathbf{T}-zI)^{-1}\mathbf{K}&=\frac{2i\left((\mu_0-z)\IM\lambda_0-2i\IM\lambda_0\IM\mu_0+(\lambda_0-z)\IM \mu_0 \right)}{(\lambda_0-z)(\mu_0-z)}\\
&=\frac{(\lambda_0-\bar\lambda_0)(\bar\mu_0-z)+(\lambda_0-z)(\mu_0-\bar\mu_0)}{(\lambda_0-z)(\mu_0-z)}.
 \end{aligned}
 $$
 Finally,
 $$
    \begin{aligned}
W_\Theta(z)&=1-2i\mathbf{K}^*(\mathbf{T}-zI)^{-1}\mathbf{K}
=1-\frac{(\lambda_0-\bar\lambda_0)(\bar\mu_0-z)+(\lambda_0-z)(\mu_0-\bar\mu_0)}{(\lambda_0-z)(\mu_0-z)}\\
&=\frac{(\lambda_0-z)(\mu_0-z)-(\lambda_0-\bar\lambda_0)(\bar\mu_0-z)-(\lambda_0-z)(\mu_0-\bar\mu_0)}{(\lambda_0-z)(\mu_0-z)}\\
&=\frac{(\bar\lambda_0-z)(\bar\mu_0-z)}{(\lambda_0-z)(\mu_0-z)}=\frac{\bar\lambda_0-z}{\lambda_0-z}\cdot\frac{\bar\mu_0-z}{\mu_0-z}
=W_{\Theta_1}(z)\cdot W_{\Theta_2}(z),
 \end{aligned}
 $$
 completing the proof.
\end{proof}

The corresponding impedance function $V_{{\Theta}}(z)$ of the coupling $\Theta$ can be found using \eqref{e6-3-6} as follows.
$$
\begin{aligned}
V_{{\Theta}}(z) &= i \frac{W_{{\Theta}} (z) - 1}{W_{{\Theta}}(z) + 1}=
i \frac{\frac{\bar\lambda_0-z}{\lambda_0-z}\cdot\frac{\bar\mu_0-z}{\mu_0-z} - 1}{\frac{\bar\lambda_0-z}{\lambda_0-z}\cdot\frac{\bar\mu_0-z}{\mu_0-z} + 1}
=i \frac{(\bar\lambda_0-z)(\bar\mu_0-z)-(\lambda_0-z)(\mu_0-z)}{(\bar\lambda_0-z)(\bar\mu_0-z)+(\lambda_0-z)(\mu_0-z)}
\\
&=i \frac{\bar\lambda_0\bar\mu_0-\bar\lambda_0 z-\bar\mu_0 z+z^2-\lambda_0\mu_0+\lambda_0 z+\mu_0 z-z^2}{\bar\lambda_0\bar\mu_0-\bar\lambda_0 z-\bar\mu_0 z+z^2+\lambda_0\mu_0-\lambda_0 z-\mu_0 z+z^2}
\\
&=i\frac{(\bar\lambda_0\bar\mu_0-\lambda_0\mu_0)+(\lambda_0-\bar\lambda_0)z-(\mu_0-\bar\mu_0)z}
{(\bar\lambda_0\bar\mu_0+\lambda_0\mu_0)-(\lambda_0+\bar\lambda_0)z-(\mu_0+\bar\mu_0)z+2z^2}\\
&=i\frac{-2i\IM(\lambda_0\mu_0)+2i(\IM\lambda_0)z+2i(\IM\mu_0)z}{2\RE(\lambda_0\mu_0)-2(\RE\lambda_0)z-2(\RE\mu_0)z+2z^2}\\
&=\frac{\IM(\lambda_0\mu_0)-(\IM\lambda_0+\IM\mu_0)z}{\RE(\lambda_0\mu_0)-(\RE\lambda_0+\RE\mu_0)z+z^2}
=\frac{-\IM(\lambda_0\mu_0)+\IM(\lambda_0+\mu_0)z}{-\RE(\lambda_0\mu_0)+\RE(\lambda_0+\mu_0)z-z^2},
\end{aligned}
$$
or
\begin{equation}\label{e-36-V}
    V_{{\Theta}}(z)=\frac{\IM(\lambda_0+\mu_0)z-\IM(\lambda_0\mu_0)}{\RE(\lambda_0+\mu_0)z-\RE(\lambda_0\mu_0)-z^2}.
\end{equation}

Now {}{we will take } one step further and show that the impedance function of the coupling of two L-systems whose individual impedances belong to bounded Donoghue classes also belongs to such a class.
\begin{theorem}[cf. \cite{BMkT-2}, \cite{MT10}]\label{t-5-Don}
 Let an L-system $\Theta$ be the coupling of two L-systems $\Theta_1$
and $\Theta_2$ of the form \eqref{e-46-Theta12} with the main operators $T_1$ and $T_2$ given by \eqref{e-45-T12}. Let also $V_{\Theta_1}(z)\in\widehat{\sM}_{\kappa_1}$ and $V_{\Theta_2}(z)\in\widehat{\sM}_{\kappa_2}$. Then the impedance function of the coupling $V_{{\Theta}}(z)$ is given by \eqref{e-36-V} belongs to the class $\widehat{\sM}_{\kappa}$ with
\begin{equation}\label{e-38-kappa}
    \kappa=\kappa_1\cdot\kappa_2.
\end{equation}
\end{theorem}
\begin{proof}
{}{Since the membership  $V_{\Theta_j}(z)\in\widehat{\sM}_{\kappa_j}$, $ j=1,2$
ensures that
 $$
 V_{\Theta_j}(i)=ia_j\quad \text{fore some}\quad  1<a_j<1, \quad  j=1,2,
 $$
 with
 $$
 \frac{1+a_j}{1-a_j}=\frac{1}{\kappa_j},\quad j=1,2,
 $$
for the transfer functions
$W_{\Theta_j}(z)$ of the systems $ \Theta_j$, $j=1,2$ we have
$$
 W_{\Theta_j}(i)=\frac{1-iV_{\Theta_j}(i)}{1+iV_{\Theta_j}(i)}=\frac{1+a_j}{1-a_j}=\frac{1}{\kappa_j},\quad j=1,2.
 $$
 Now  \eqref{e-38-kappa} is an immediate consequence of the multiplication Theorem \ref{unitar}:
 $$
\kappa_1\cdot\kappa_2=\frac{1}{W_{\Theta_1}(i)}\cdot\frac{1}{W_{\Theta_2}(i)}=\frac{1}{W_{\Theta}(i)}=\kappa.
$$
}
 \end{proof}

\begin{remark}
{}{
Alternatively one can argue  as follows. We have shown in Theorem \ref{t-6}, {that the requirement that} $V_{\Theta_1}(z)\in\widehat{\sM}_{\kappa_1}$ and $V_{\Theta_2}(z)\in\widehat{\sM}_{\kappa_2}$ is equivalent to $\lambda_0=a_1 i$ and $\mu_0=a_2 i$ for some $0<a_1<1$ and $0<a_2<1$. It has been already shown above  that $V_{{\Theta}}(z)$ is given by \eqref{e-36-V} and hence
$$
V_{{\Theta}}(i)=\frac{\IM(a_1 i+a_2 i)i-\IM(-a_1a_2)}{\RE(a_1 i+a_2 i)i-\RE(-a_1a_2)+1}=
\frac{a_1 +a_2}{1+a_1 a_2}\,i=ai,
$$
where
$$a=\frac{a_1 +a_2}{1+a_1 a_2}.$$ Applying Theorem \ref{t-6} again we conclude that $V_{{\Theta}}(z)$ belongs to one of the bounded Donoghue classes. All that remains is to find the value of corresponding {parameter} $\kappa$.
Since $0<a_1<1$ and $0<a_2<1$ {implies} $\frac{a_1 +a_2}{1+a_1 a_2}<1$,  we apply \eqref{e-45-kappa-1} to get
$$
\kappa=\frac{1-a}{1+a}=\frac{1-\frac{a_1 +a_2}{1+a_1 a_2}}{1+\frac{a_1 +a_2}{1+a_1 a_2}}=\frac{1+a_1a_2-a_1-a_2}{1+a_1a_2+a_1+a_2}=\frac{(1-a_1)(1-a_2)}{(1+a_1)(1+a_2)}=\kappa_1\cdot\kappa_2.
$$
}
\end{remark}

Analogous  results take place for $V_{\Theta_1}(z)\in\widehat{\sM_{\kappa_1}^{-1}}$ and $V_{\Theta_2}(z)\in\widehat{\sM_{\kappa_2}^{-1}}$ as well as if  both $V_{\Theta_1}(z),\, V_{\Theta_2}(z)\in\widehat{\sM}$.

 The obtained result is useful to compare with the multiplicativity of the von Neumann parameters obtained in \cite[Theorem 5.4]{MT10} when discussing the coupling of unbounded dissipative  operators.

\section{c-Entropy and dissipation coefficient of an L-system}\label{s6}

In this section we are going to evaluate the c-entropy and dissipation coefficient of an
{}{L-system with   multiplication operator.}

We begin with reminding {}{the definition }
 of the c-entropy of an L-system introduced in \cite{BT-16}.
\begin{definition}\label{d-9-ent}
Let $\Theta$ be an L-system of the form \eqref{e-4-34}. The quantity
\begin{equation}\label{e-80-entropy-def}
    \calS=-\ln (|W_\Theta(-i)|),
\end{equation}
where $W_\Theta(z)$ is the transfer function of $\Theta$, is called the \textbf{coupling entropy} (or \textbf{c-entropy}) of the L-system $\Theta$.
\end{definition}
 Note that if, in addition,  the point $z=i$ belongs to $\rho(T)$, then we also have that
\begin{equation}\label{e-80-entropy}
     \calS=\ln (|W_\Theta(i)|).
\end{equation}

{}{The c-entropy of an L-system of the form \eqref{e-4-34} with  multiplication operator \eqref{e-d-4-30} can explicitly be evaluates as follows.}
\begin{theorem}\label{t-8}%
Let $\Theta$ be an L-system   of the form \eqref{e-4-34} that is based upon multiplication operator \eqref{e-d-4-30}. Then  the c-entropy $\calS$ of $\Theta$ is
\begin{equation}\label{e-40-entropy}
     \calS=\frac{1}{2}\ln{\frac{(\RE\lambda_0)^2+(1+\IM\lambda_0)^2}{(\RE\lambda_0)^2+(1-\IM\lambda_0)^2)}}.
\end{equation}
\end{theorem}
\begin{proof}
Our plan is to find $\calS$ by the definition using \eqref{e-80-entropy-def}. Utilizing \eqref{e-4-35} we obtain
$$
|W_\Theta(-i)|=\left|\frac{\bar\lambda_0+i}{\lambda_0+i}\right|
=\left|\frac{\RE\lambda_0+i(1-\IM\lambda_0)}{\RE\lambda_0+i(1+\IM\lambda_0)}\right|
=\sqrt{\frac{(\RE\lambda_0)^2+(1-\IM\lambda_0)^2}{(\RE\lambda_0)^2+(1+\IM\lambda_0)^2}}.
$$
{}{Therefore,}
\begin{equation}\label{e-41}
    \begin{aligned}
\calS&=-\ln (|W_\Theta(-i)|)=-\ln\sqrt{\frac{(\RE\lambda_0)^2+(1-\IM\lambda_0)^2}{(\RE\lambda_0)^2+(1+\IM\lambda_0)^2}}
\\
&=\frac{1}{2}\ln\left( \frac{(\RE\lambda_0)^2+(1+\IM\lambda_0)^2}{(\RE\lambda_0)^2+(1-\IM\lambda_0)^2}\right).
    \end{aligned}
\end{equation}
The proof is complete.
\end{proof}

\begin{remark}\label{r-11}
{}{As} it follows from Theorems \ref{t-6}, \ref{t-8}, and formula \eqref{e-40-entropy}, an L-system of the form \eqref{e-4-34} achieves {}{its}  infinite c-entropy if and only $\lambda_0=i$. {}{Recall  that a finite L-system's c-entropy {}{attains} its maximum whenever its impedance function $V_\Theta(z)$ belongs to one of the generalized Donoghue classes \cite{BT-22}. }That is,
$V_\Theta(i)$ {}{is purely  imaginary.}
{}{In the current setting,
according to \eqref{e-4-37},   $\RE(V_\Theta(i))$ vanishes if and only if $\RE\lambda_0=0$.} Therefore,
c-entropy
of an L-system of the form \eqref{e-4-34} attains its  maximum (finite) value whenever  $\lambda_0=ai$, for $a\in(0,1)\cup(1,+\infty)$.
\end{remark}

 The graph of c-entropy {}{$\calS(x,y)$ }
as a function of $x=\RE\lambda_0$ and $y=\RE\lambda_0$ is shown on Figure \ref{fig-1}. The visible pick is actually
{}{represents}
the infinite value of {}{$\calS(x,y)$ } that occurs at the point $(0,1)$ when $\lambda_0=i$.

\begin{figure}
  \begin{center}
 \includegraphics[width=90mm]{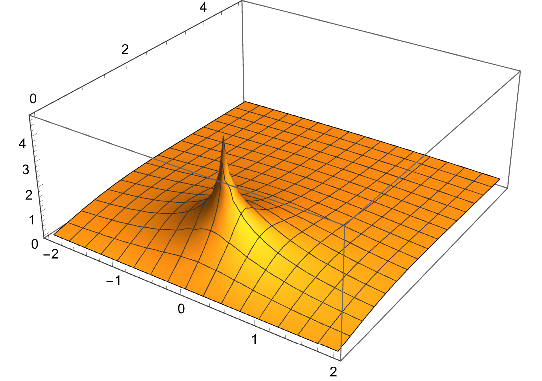}
  \caption{Maximal c-Entropy }\label{fig-1}
  \end{center}
\end{figure}

Let us recall the definition of the dissipation coefficient of an L-system.
\begin{definition}[{cf. \cite{BT-16}}, \cite{BT-21}]\label{d-10}
Let $\Theta$ be an L-system  of the form \eqref{e-4-34} with c-entropy $\calS$. Then
the quantity
 \begin{equation}\label{e-69-ent-dis}
\calD=1-e^{-2\calS}
\end{equation}
 is called the \textbf{coefficient of dissipation} (or dissipation coefficient) of the L-system $\Theta$.
\end{definition}

\begin{theorem}\label{t-9}%
Let $\Theta$ be an L-system   of the form \eqref{e-4-34}  {}{with}  multiplication operator \eqref{e-d-4-30}. Then  the dissipation coefficient $\calD$ of $\Theta$ is
\begin{equation}\label{e-44-dc}
     \calD=\frac{4\IM\lambda_0}{(\RE\lambda_0)^2+(1+\IM\lambda_0)^2}.
\end{equation}
\end{theorem}
\begin{proof}
In order to simplify calculations let us denote $\lambda_0=a+bi$. Using \eqref{e-69-ent-dis} with the middle part of \eqref{e-41} we get
$$
\begin{aligned}
\calD&=1-e^{-2\cS}=1-\left|\frac{\bar\lambda_0+i}{\lambda_0+i}\right|^2=1-\frac{(\RE\lambda_0)^2+(1-\IM\lambda_0)^2}{(\RE\lambda_0)^2+(1+\IM\lambda_0)^2}\\
&=1-\frac{a^2+(1-b)^2}{a^2+(1+b)^2}=\frac{a^2+(1+b)^2-a^2-(1-b)^2}{a^2+(1+b)^2}=\frac{4b}{a^2+(1+b)^2}.
\end{aligned}
$$
Thus, \eqref{e-44-dc} takes place.
\end{proof}

Now, we turn our attention to the c-entropy of L-system coupling, as described in Section \ref{s5}. The following theorem establishes the additivity property of c-Entropy with respect to the coupling of two L-systems, thereby justifying the use of the term “coupling entropy.”

\begin{theorem}\label{t-13}
Let an L-system $\Theta$ be the coupling of two L-systems $\Theta_1$ and $\Theta_2$ of the form \eqref{e-46-Theta12} with the corresponding c-entropies $\calS_1$ and $\calS_2$.
Then the c-entropy $\calS$ of $\Theta$ is such that
\begin{equation}\label{e-56-ent}
\calS=\calS_1+\calS_2.
\end{equation}

If either $\calS_1=\infty$ or $\calS_2=\infty$, then $\calS=\infty$.
\end{theorem}
\begin{proof}
We rely on the Definition \ref{d-9-ent} of c-entropy and Theorem \ref{unitar}. Applying \eqref{e-80-entropy-def} with \eqref{e-91-mult} yields
$$
\begin{aligned}
\calS&=-\ln (|W_\Theta(-i)|)=-\ln (|W_{\Theta_1}(-i)\cdot W_{\Theta_2}(-i)|)\\
&=-\ln (|W_{\Theta_1}(-i)|-\ln (|W_{\Theta_2}(-i)|=\calS_1+\calS_2.
\end{aligned}
$$
\end{proof}
It follows directly from Theorem \ref{t-9} and formulas \eqref{e-40-entropy}-\eqref{e-41} that the c-entropy of the coupling of two L-systems $\Theta_1$ and $\Theta_2$ of the form \eqref{e-46-Theta12} can be written in terms of the defining numbers $\lambda_0$ and $\mu_0$ as follows.
\begin{equation}\label{e-57}
    \begin{aligned}
\calS&=\frac{1}{2}\ln\left( \frac{(\RE\lambda_0)^2+(1+\IM\lambda_0)^2}{(\RE\lambda_0)^2+(1-\IM\lambda_0)^2}\right)+
\frac{1}{2}\ln\left( \frac{(\RE\mu_0)^2+(1+\IM\mu_0)^2}{(\RE\mu_0)^2+(1-\IM\mu_0)^2}\right)
\\
&=\frac{1}{2}\ln\left( \frac{(\RE\lambda_0)^2+(1+\IM\lambda_0)^2}{(\RE\lambda_0)^2+(1-\IM\lambda_0)^2}
\cdot\frac{(\RE\mu_0)^2+(1+\IM\mu_0)^2}{(\RE\mu_0)^2+(1-\IM\mu_0)^2}\right).
    \end{aligned}
\end{equation}

Now we are going to look into the dissipation coefficient of the coupling of two L-systems of the form \eqref{e-46-Theta12}. In \cite{BT-21} we made a note that  if L-system $\Theta$ with c-entropy $\calS$ is a coupling of two L-systems $\Theta_1$ and $\Theta_2$ with c-entropies $\calS_1$ and $\calS_2$, respectively, formula \eqref{e-56-ent} holds, i.e., $\calS=\calS_1+\calS_2.$
Let $\calD$, $\calD_1$, and $\calD_2$ be the dissipation coefficients of L-systems $\Theta$, $\Theta_1$, and $\Theta_2$. Then \eqref{e-69-ent-dis} implies
$$
1-\calD=e^{-2\cS}=e^{-2(\cS_1+\cS_2)}=e^{-2\cS_1}\cdot e^{-2\cS_2}=(1-\calD_1)(1-\calD_2).
$$
Thus, the formula
\begin{equation}\label{e-72-coupling}
    \calD=1-(1-\calD_1)(1-\calD_2)=\calD_1+\calD_2-\calD_1\calD_2
\end{equation}
describes the coefficient of dissipation of the L-system coupling. The following theorem gives the value of the dissipation coefficient of two L-systems of the form \eqref{e-46-Theta12} in terms of their defining numbers $\lambda_0$ and $\mu_0$.
\begin{theorem}\label{t-14}
Let an L-system $\Theta$ be the coupling of two L-systems $\Theta_1$ and $\Theta_2$ of the form \eqref{e-46-Theta12} with the corresponding coefficients of dissipation  $\calD_1$ and $\calD_2$.
Then the dissipation coefficient $\calD$ of $\Theta$ is such that
\begin{equation}\label{e-59-D}
\calD=\frac{4\IM\lambda_0(|\mu_0|^2+1)+4\IM\mu_0(|\lambda_0|^2+1)}{[(\RE\lambda_0)^2+(1+\IM\lambda_0)^2][(\RE\mu_0)^2+(1+\IM\mu_0)^2]}.
\end{equation}
\end{theorem}
\begin{proof}
Applying formula \eqref{e-72-coupling} with \eqref{e-44-dc} gives
$$
     \begin{aligned}
\calD&=\frac{4\IM\lambda_0}{(\RE\lambda_0)^2+(1+\IM\lambda_0)^2}+\frac{4\IM\mu_0}{(\RE\mu_0)^2+(1+\IM\mu_0)^2}\\
&\quad+\frac{16\IM\lambda_0\cdot\IM\mu_0}{[(\RE\lambda_0)^2+(1+\IM\lambda_0)^2][(\RE\mu_0)^2+(1+\IM\mu_0)^2]}\\
&=\frac{4\IM\lambda_0(|\mu_0|^2+1)+4\IM\mu_0(|\lambda_0|^2+1)}{[(\RE\lambda_0)^2+(1+\IM\lambda_0)^2][(\RE\mu_0)^2+(1+\IM\mu_0)^2]}.
     \end{aligned}
$$
Thus \eqref{e-59-D} takes place.
\end{proof}

\section{Skew-adjoint L-system with multiplication operator}\label{s7}

Let (as above) $\calH$ be  a one-dimensional Hilbert space with an inner product $(\cdot,\cdot)$ and a normalized basis vector $h_0\in \calH$, ($\|h_0\|=1$), and  $T$  be a multiplication operator in $\calH$ of the form \eqref{e-d-4-30}. Consider the skew-adjoint to $T$ operator $T^\times$ defined by the formula
\begin{equation}\label{e-d-60}
    T^\times h=-\bar\lambda_0 h,\quad h\in\calH.
\end{equation}
Clearly, for all $h\in\calH$ we have $(T h,h)=-(h,T^\times h)$ and
$$
(T^\times)^* h=-T h=-\lambda_0 h,\quad\IM T^\times h=\IM\lambda_0 h,\quad\textrm{ and }\quad \RE T^\times  h=-\RE \lambda_0 h.
$$
We would like to include $T^\times$ into a  L-system $\Theta^\times$ of the form \eqref{e-4-34}. Let $K:\dC\rightarrow\calH$ be  of the form \eqref{e-4-31} with $K^*:\calH\rightarrow\dC$ given by \eqref{e-4-32} exactly as described in Section \ref{s4}. Then,
\begin{equation}\label{e-61}
    \IM T^\times h=\IM\lambda_0 h=\IM T h=K\,K^* h,\; \forall h\in\calH,
\end{equation}
and hence $T^\times$ can be included into the L-system
\begin{equation}\label{e-7-61}
    \Theta^\times= \begin{pmatrix} T^\times&K&\ 1\cr  \calH & &\dC\cr \end{pmatrix},
\end{equation}
where operators $T$ and $K$ are defined by \eqref{e-d-60} and \eqref{e-4-31}, respectively. The L-system $\Theta^\times$ of the form \eqref{e-7-61} will be called \textbf{skew-adjoint L-system} with respect to an L-system $\Theta$ of the form \eqref{e-4-34}.

Repeating the steps described in Section \ref{s4}, one relies on \eqref{e-4-35} to obtain
\begin{equation}\label{e-7-63}
    W_{\Theta^\times} (z)=I-2iK^\ast (T^\times-zI)^{-1}K=\frac{\lambda_0+z}{\bar\lambda_0+z}.
\end{equation}
The corresponding impedance function  is easily found and given by
\begin{equation}\label{e-7-64}
    V_{\Theta^\times} (z)=K^\ast (\RE T^\times - zI)^{-1} K=-\frac{\IM\lambda_0}{\RE\lambda_0+z}.
\end{equation}
\begin{theorem}\label{t-15}
Let an L-system $\Theta$ be  of the form \eqref{e-4-34} with the corresponding skew-adjoint L-system $\Theta^\times$ of the form \eqref{e-7-61}.
Then the c-entropies $\calS(\Theta)$ and $\calS(\Theta^\times)$ and dissipation coefficients $\calD(\Theta)$ and $\calD(\Theta^\times)$ of $\Theta$ and $\Theta^\times$ coincide, respectively.
\end{theorem}
\begin{proof}
We are going to find the c-entropy  $\calS(\Theta^\times)$ by the definition using \eqref{e-80-entropy-def} and \eqref{e-4-35}. We obtain
$$
|W_{\Theta^\times}(-i)|=\left|\frac{\lambda_0-i}{\bar\lambda_0-i}\right|
=\left|\frac{\RE\lambda_0-i(1-\IM\lambda_0)}{\RE\lambda_0-i(1+\IM\lambda_0)}\right|
=\sqrt{\frac{(\RE\lambda_0)^2+(1-\IM\lambda_0)^2}{(\RE\lambda_0)^2+(1+\IM\lambda_0)^2}}.
$$
But we have shown in the proof of Theorem \ref{t-8}
$$
|W_\Theta(-i)|=\left|\frac{\bar\lambda_0+i}{\lambda_0+i}\right|
=\left|\frac{\RE\lambda_0+i(1-\IM\lambda_0)}{\RE\lambda_0+i(1+\IM\lambda_0)}\right|
=\sqrt{\frac{(\RE\lambda_0)^2+(1-\IM\lambda_0)^2}{(\RE\lambda_0)^2+(1+\IM\lambda_0)^2}}.
$$
Thus, $|W_\Theta(-i)|=|W_{\Theta^\times}(-i)|$ and, applying the definition of c-entropy \eqref{e-80-entropy-def}, one obtains that the c-entropies of both L-systems $\Theta$ and $\Theta^\times$ match, that is $\calS(\Theta)=\calS(\Theta^\times)$.

The equality of the dissipation coefficients $\calD(\Theta)$ and $\calD(\Theta^\times)$ follows immediately from \eqref{e-69-ent-dis}.
\end{proof}

Let us now consider the coupling of the L-systems $\Theta$ and $\Theta^\times$. Hence, we are applying  Definition \ref{d-11} to the L-systems of the form \eqref{e-4-34}
\begin{equation*}\label{e-65-Theta12}
    \Theta= \begin{pmatrix} T&K&\ 1\cr  \calH & &\dC\cr \end{pmatrix}
\quad \textrm{and}\quad
\Theta^\times= \begin{pmatrix} T^\times&K&\ 1\cr  \calH & &\dC\cr \end{pmatrix},
\end{equation*}
where
\begin{equation}\label{e-66-T12}
    T h=\lambda_0 h\quad \textrm{and}\quad T^\times h=-\bar\lambda_0 h,\quad h\in\calH,
\end{equation}
and
where $K:\dC\rightarrow\calH$ and $K^*:\calH\rightarrow\dC$
are such that
\begin{equation*}\label{e-67-K12}
   K c=(\sqrt{\IM\lambda_0}\, c) h_0
   \quad \textrm{and}\quad
    K^* h=(h, \sqrt{\IM\lambda_0}\,h_0),
   \quad \forall c\in\dC.
\end{equation*}
Here $h_0\in \calH$ is normalized basis vector in $\calH$, ($\|h_0\|=1$). Let
$$
\mathbf{\Theta}=\Theta\cdot\Theta^\times.
$$
Then
 \begin{equation*}
\mathbf{\Theta}= \begin{pmatrix} \mathbf{T}&\mathbf{K}&\ 1\cr \calH\oplus\calH& &\dC\cr
\end{pmatrix}
\end{equation*}
where $\mathbf{T}$ is of the form \eqref{e-48-bT} and hence, taking into account \eqref{e-66-T12},
\begin{equation}\label{e-65-bT}
  \begin{aligned}
   \mathbf{T}\left[
               \begin{array}{c}
                 h_1 \\
                 h_2 \\
               \end{array}
             \right]&=\left(
                       \begin{array}{cc}
                         T & 2iKK^* \\
                         0 & T^\times \\
                       \end{array}
                     \right)\left[ \begin{array}{c}
                 h_1 \\
                 h_2 \\
               \end{array}
             \right]=\left[
               \begin{array}{c}
                 T h_1+2i KK^*h_2 \\
                 T^\times h_2 \\
               \end{array}
             \right]\\
             &=\left[
               \begin{array}{c}
                 \lambda_0 h_1+2i  \sqrt{(\IM\lambda_0)^2}\,h_2\\
                 -\bar\lambda_0 h_2 \\
               \end{array}
             \right]
             =\left[
               \begin{array}{c}
                 \lambda_0 h_1+(\lambda_0-\bar\lambda_0)\,h_2\\
                 -\bar\lambda_0 h_2 \\
               \end{array}
             \right],
      \end{aligned}
   \end{equation}
for all $h_1,h_2\in\calH$. Also, (see \eqref{e-50-bK})
\begin{equation}\label{e-67-bK}
    \mathbf{K}c=\left[
               \begin{array}{c}
                 K c \\
                 K c\\
               \end{array}
             \right]=\left[
               \begin{array}{c}
                 (\sqrt{\IM\lambda_0}\, c) h_0 \\
                (\sqrt{\IM\lambda_0}\, c) h_0\\
               \end{array}
             \right].
\end{equation}
The transfer function $W_{\mathbf{\Theta}}(z)$ of this coupling can be found using Theorem \ref{unitar} with formulas \eqref{e-91-mult} and \eqref{e-7-63}
$$
W_{\mathbf{\Theta}}(z)=W_{\Theta}(z)\cdot W_{\Theta^\times}(z)=\frac{\bar\lambda_0-z}{\lambda_0-z}\cdot\frac{\lambda_0+z}{\bar\lambda_0+z}
=\frac{|\lambda_0|^2-(\lambda_0-\bar\lambda_0)z-z^2}{|\lambda_0|^2+(\lambda_0-\bar\lambda_0)z-z^2}
$$
or
\begin{equation}\label{e-68}
    W_{\mathbf{\Theta}}(z)=\frac{|\lambda_0|^2-2i(\IM\lambda_0) z-z^2}{|\lambda_0|^2+2i(\IM\lambda_0)z-z^2}.
\end{equation}
The impedance function $V_{\mathbf{\Theta}}(z)$ of this coupling can be found using \eqref{e6-3-6}. We have
$$
\begin{aligned}
V_{\mathbf{\Theta}}(z) &= i \frac{W_{\mathbf{\Theta}} (z) - 1}{W_{\mathbf{\Theta}}(z) + 1}=
i \frac{\frac{|\lambda_0|^2-2i(\IM\lambda_0) z-z^2}{|\lambda_0|^2+2i(\IM\lambda_0)z-z^2} - 1}{\frac{|\lambda_0|^2-2i(\IM\lambda_0) z-z^2}{|\lambda_0|^2+2i(\IM\lambda_0)z-z^2} + 1}\\
&=i \frac{|\lambda_0|^2-2i(\IM\lambda_0)z-z^2-|\lambda_0|^2-2i(\IM\lambda_0)z+z^2}{|\lambda_0|^2-2i(\IM\lambda_0) z-z^2+|\lambda_0|^2+2i(\IM\lambda_0)z-z^2}=i\frac{-2i(\IM\lambda_0)z}{|\lambda_0|^2-z^2},
\end{aligned}
$$
or
\begin{equation}\label{e-69-V}
    V_{\mathbf{\Theta}}(z)=\frac{2(\IM\lambda_0)z}{|\lambda_0|^2-z^2}.
\end{equation}

In particular,
\begin{equation}\label{e-70-Vi}
    V_{\mathbf{\Theta}}(i)=\frac{2(\IM\lambda_0)}{|\lambda_0|^2+1}\,i,
\end{equation}
and hence the impedance function $V_{\mathbf{\Theta}}(z)$ always belongs to one of the bounded Donoghue classes discussed in Section \ref{s3}. For instance, if $\lambda_0=i$, then
$$V_{\mathbf{\Theta}}(z)\in \widehat{\sM}.$$

\begin{remark}\label{zepi2}
Taking into account that the impedance function  $V_{\mathbf{\Theta}}(z)$ for  the coupling of two elementary systems $\Theta$ and $\Theta^\times$ is expressed as
$$
  V_{\mathbf{\Theta}}(z)=\frac{2(\IM\lambda_0)z}{|\lambda_0|^2-z^2}=\frac{\Im \lambda_0}{|\lambda_0|-z}+\frac{\Im \lambda_0}{
 ( -|\lambda_0|)-z},
  $$
  we can use  Remark \ref{zepi} to draw the following conclusion:  the function
  $$
  Z_{\bf \Theta}(p)=\frac1iV_{\mathbf{\Theta}}(ip)
  $$
  coincide with the impedance of the parallel $LC$-circuit with
  $$L=\frac{\IM\lambda_0}{|\lambda_0|^2}=\Im \left (-\frac1{\lambda_0}\right )
  \quad\text{and}\quad
  C=\frac{1}{\IM\lambda_0},
  $$
  the inductance and capacitance, respectively. The discussion above sheds light on the``internal''  structure of the  sequential connection of oscillatory $LC$-circuits mentioned in  Remark \ref{zepi}:
   each element of such a circuit  has the impedance of the coupling of two elementary systems.
\end{remark}

The following Corollary immediately follows from Theorems \ref{t-13} and \ref{t-15}.
\begin{corollary}\label{c-16}
Let an L-system $\Theta$ be  of the form \eqref{e-4-34} with the c-entropy $\calS$ and dissipation coefficient $\calD$. Let also $\Theta^\times$ be the corresponding to $\Theta$ skew-adjoint L-system  of the form \eqref{e-7-61}.
Then the c-entropy $\calS(\mathbf{\Theta})$ of the coupling $\mathbf{\Theta}=\Theta\cdot\Theta^\times$ is
\begin{equation}\label{e-69}
    \calS(\mathbf{\Theta})=2\calS.
\end{equation}
Moreover, the dissipation coefficient $\calD(\mathbf{\Theta})$ of this coupling is
\begin{equation}\label{e-70}
    \calD(\mathbf{\Theta})=2\calD-\calD^2.
\end{equation}
\end{corollary}

Table \ref{Table-2} summarizes the results from Sections  \ref{s4}--\ref{s7} that relate to c-entropy and dissipation coefficient for the L-systems of the form \eqref{e-4-34} with $\lambda_0=x+i y$. In this table we assume that the L-systems $\Theta$, $\Theta_1$,  $\Theta_2$ of the form \eqref{e-4-34} have c-entropies $\calS$, $\calS_1$, $\calS_2$ and the dissipation coefficients $\calD$, $\calD_1$, $\calD_2$, respectively.

\begin{table}[ht]
\centering
\begin{tabular}{|c|c|c|c|}
\hline
 &  &  &\\
 \textbf{L-system}& \textbf{c-entropy} $\calS$  & \textbf{Dissipation}  & \textbf{Theorems}  \\
  &  & \textbf{coefficient} $\calD$&\\ \hline
&  &  &\\
  $\Theta$ & $\ln\sqrt{\frac{x^2+(1+y)^2}{x^2+(1-y)^2}}$ & $\frac{4y}{x^2+(1+y)^2}$ &Theorems \ref{t-8}\\
  &  &  & and \ref{t-9}\\  \hline
 &  &  &\\
 $\Theta^\times$&  $\ln\sqrt{\frac{x^2+(1+y)^2}{x^2+(1-y)^2}}$&  $\frac{4y}{x^2+(1+y)^2}$ &Theorem \ref{t-15}\\
 &  &  & \\
  \hline
 &  &  &\\
 $\Theta=\Theta_1\cdot\Theta_2$  &  $\calS=\calS_1+\calS_2$& Formula \eqref{e-59-D} &Theorems \ref{t-13}\\
  &  &  & and \ref{t-14}\\
     \hline
      &  &  &\\
 $\mathbf{\Theta}=\Theta\cdot\Theta^\times$  & $\calS(\mathbf{\Theta})=2\calS $ & $\calD(\mathbf{\Theta})=2\calD-\calD^2$ & Corollary \ref{c-16}\\
  &  &  & \\
     \hline
     \multicolumn{1}{l}{} & \multicolumn{1}{l}{} & \multicolumn{1}{l}{} & \multicolumn{1}{l}{}
\end{tabular}
\caption{c-Entropy and Dissipation coefficient  with $\lambda_0=x+i y$}
\label{Table-2}
\end{table}

\section{Examples}

In this section we present  examples that illustrate the construction of L-system of the form \eqref{e-7-61} for different values of $\lambda_0$.

\subsection*{Example 1}\label{ex-1}

Let $\lambda_0=i$ and consider the linear operator $T$ of the form \eqref{e-d-4-30} acting on  an arbitrary one-dimensional Hilbert space with an inner product $(\cdot,\cdot)$ and a normalized basis vector $h_0\in \calH$, ($\|h_0\|=1$). Then
\begin{equation}\label{e-ex1-72}
    T h=i\, h,\quad h\in\calH.
\end{equation}
Clearly,
$$
T^* h=(-i)\, h,\quad\IM T h= h,\quad\textrm{ and }\quad \RE T h=0 h=0,\quad \forall h\in\calH.
$$
We are going to include $T$ into an L-system $\Theta$ of the form \eqref{e-d-4-30}. In order to do that we take an operator $K:\dC\rightarrow\calH$ of the form
\begin{equation}\label{e-ex1-73}
   K c= c\, h_0,\quad \forall c\in\dC.
\end{equation}
Then, the adjoint operator $K^*:\calH\rightarrow\dC$ is
\begin{equation}\label{e-ex1-74}
   K^* h=(h, h_0),\quad \forall h\in\calH.
\end{equation}
Consequently,
$
K\,K^* h=h.
$
We are constructing an L-system of the form \eqref{e-d-4-30}
\begin{equation}\label{e-ex1-75}
    \Theta= \begin{pmatrix} T&K&\ 1\cr  \calH & &\dC\cr \end{pmatrix},
\end{equation}
where operators $T$ and $K$ are defined by \eqref{e-ex1-72}--\eqref{e-ex1-74}, respectively. Using \eqref{e-4-35}  we have
\begin{equation}\label{e-ex1-76}
    W_\Theta (z)=I-2iK^\ast (T-zI)^{-1}K=\frac{-i-z}{i-z}=\frac{z+i}{z-i}.
\end{equation}
The corresponding impedance function is easily found using \eqref{e-4-36}
\begin{equation}\label{e-ex1-77}
    V_\Theta (z)=K^\ast (\RE T - zI)^{-1} K=-\frac{1}{z}.
\end{equation}
Clearly, $V_\Theta (z)\in\widehat{\sM}$. Also, the c-entropy $\calS$ of the L-system $\Theta$ in \eqref{e-ex1-75} is infinity while the dissipation coefficient $\calD$ is $1$. Moreover, since $i=-\bar i$, the skew-adjoint system $\Theta^\times$ in this case coincides with $\Theta$ making it ``self-skew-adjoint".

\subsection*{Example 2}\label{ex-2}

Following the steps of Example 1 we let $\lambda_0=1+i$ and consider the linear operator $T$ of the form \eqref{e-d-4-30} acting on  an arbitrary one-dimensional Hilbert space with an inner product $(\cdot,\cdot)$ and a normalized basis vector $h_0\in \calH$, ($\|h_0\|=1$). Then
\begin{equation}\label{e-ex2-72}
    T h=(1+i)\, h,\quad h\in\calH.
\end{equation}
Clearly,
$$
T^* h=(1-i)\, h,\quad\IM T h= h,\quad\textrm{ and }\quad \RE T h=h,\quad \forall h\in\calH.
$$
We are going to include $T$ into an L-system $\Theta$ of the form \eqref{e-d-4-30}. In order to do that we take an operator $K:\dC\rightarrow\calH$ of the form
\begin{equation}\label{e-ex2-73}
   K c= c\, h_0,\quad \forall c\in\dC.
\end{equation}
Then, the adjoint operator $K^*:\calH\rightarrow\dC$ is
\begin{equation}\label{e-ex2-74}
   K^* h=(h, h_0),\quad \forall h\in\calH.
\end{equation}
Consequently,
$
K\,K^* h=h.
$
We are constructing an L-system of the form \eqref{e-d-4-30}
\begin{equation}\label{e-ex2-75}
    \Theta= \begin{pmatrix} T&K&\ 1\cr  \calH & &\dC\cr \end{pmatrix},
\end{equation}
where operators $T$ and $K$ are defined by \eqref{e-ex2-72}--\eqref{e-ex2-74}, respectively. Using \eqref{e-4-35}  we have
\begin{equation}\label{e-ex2-76}
    W_\Theta (z)=I-2iK^\ast (T-zI)^{-1}K=\frac{1-i-z}{1+i-z}.
\end{equation}
The corresponding impedance function is easily found using \eqref{e-4-36}
\begin{equation}\label{e-ex2-77}
    V_\Theta (z)=K^\ast (\RE T - zI)^{-1} K=\frac{1}{1-z}.
\end{equation}
The c-entropy $\calS$ of the L-system $\Theta$ in \eqref{e-ex2-75} is found via \eqref{e-80-entropy-def} and \eqref{e-ex2-76} as follows
\begin{equation}\label{e-ex2-84}
    \calS=-\ln (|W_\Theta(-i)|)=-\ln\left|\frac{1}{1+2i}\right|=\ln (\sqrt{5})=\frac{1}{2}\ln5.
\end{equation}
The corresponding coefficient of dissipation $\calD$ is (see \eqref{e-69-ent-dis})
\begin{equation}\label{e-ex2-85}
\calD=1-e^{-2\cS}=1-e^{-2(\frac{1}{2}\ln5)}=1-\frac{1}{5}=\frac{4}{5}.
\end{equation}
Note, that even though the L-systems \eqref{e-ex1-75} and \eqref{e-ex2-75} share the same value of $\IM\lambda_0$ and (as a result) channel operators $K$, the c-entropy in Example 2 is finite.

Now let us construct the skew-adjoint system $\Theta^\times$. By the definition \eqref{e-d-60} we have
\begin{equation}\label{e-ex2-86}
    T^\times h=-\overline{(1+i)}\, h=(-1+i)\, h,\quad h\in\calH
\end{equation}
and the L-system
\begin{equation}\label{e-ex2-87}
    \Theta^\times= \begin{pmatrix} T^\times&K&\ 1\cr  \calH & &\dC\cr \end{pmatrix},
\end{equation}
where $K$ is defined by \eqref{e-ex2-73} is skew-adjoint to $\Theta$. Using \eqref{e-7-63} we have
\begin{equation}\label{e-ex2-88}
    W_{\Theta^\times} (z)=I-2iK^\ast (T^\times-zI)^{-1}K=\frac{1+i+z}{1-i+z}.
\end{equation}
The corresponding impedance function is easily found using \eqref{e-7-64} and is
\begin{equation}\label{e-ex2-89}
    V_{\Theta^\times} (z)=K^\ast (\RE T^\times - zI)^{-1} K=-\frac{1}{1+z}.
\end{equation}
Both L-systems $\Theta$ and $\Theta^\times$ share c-entropy $\calS$ and dissipation coefficient $\calD$ values given by \eqref{e-ex2-84} and \eqref{e-ex2-85}, respectively.


\end{document}